\documentclass[a4paper,11pt]{amsproc}

\usepackage{xargs}
\usepackage{mathtools, todonotes}
\usepackage{amsmath,amssymb,amsfonts,amsthm}

\usepackage[colorlinks, backref]{hyperref}
\usepackage{amsrefs}
\usepackage{enumerate, tikz-cd, nicefrac}

\title{High-dimensional expansion and soficity of groups}

\author{Lukas Gohla}
\address{Lukas Gohla, TU Dresden, 01062 Dresden, Germany}
\email{lukas.gohla@posteo.net}
\author{Andreas Thom}
\address{Andreas Thom, TU Dresden, 01062 Dresden, Germany}
\email{andreas.thom@tu-dresden.de}

\theoremstyle{plain}
\newtheorem{theorem}{Theorem}[section]
\newtheorem{definition}[theorem]{Definition}

\newtheorem{lemma}[theorem]{Lemma}

\newtheorem{corollary}[theorem]{Corollary}

\newtheorem{question}[theorem]{Question}

\theoremstyle{definition}
\newtheorem{remark}[theorem]{Remark}

\newcommand{\vertiii}[1]{{|\kern-0.2ex|\kern-0.2ex| #1 
    |\kern-0.2ex|\kern-0.2ex|}}

\begin{document}
\begin{abstract}
For $d \geq 4$ and $p$ a sufficiently large prime, we construct a lattice $\Gamma \leq {\rm PSp}_{2d}(\mathbb Q_p),$ such that its universal central extension cannot be sofic if $\Gamma$ satisfies some weak form of stability in permutations. In the proof, we make use of high-dimensional expansion phenomena and, extending results of Lubotzky, we construct new examples of cosystolic expanders over arbitrary finite abelian groups.
\end{abstract}

\maketitle

\tableofcontents

\section{Introduction}

Eversince the influential paper of Gromov \cite{MR1694588} and subsequent work of Weiss \cite{MR1803462}, the quest for a non-sofic group has inspired mathematicians. In order to show that a group is \emph{sofic}, one has to provide sufficiently rich almost representations of the group in permutations. The competing notion of \emph{stability} requires that every sufficiently accurate almost representation in permutations is close to an actual permutation representation. Thus, it is easy to see that any group which is both sofic and stable must be residually finite and, at least theoretically, this opens a route to the construction of non-sofic groups, see \cite{MR2500002}. Unfortunately, stability is a rare phenomenon, even though there is some indication that lattices in algebraic groups of rank at least $3$ share the right kind of rigidity in order to ensure at least some form of stability. This was first discovered in \cite{MR4080477} and related to high-dimensional expansion, see also \cite{MR3966829}, but the techniques are currently not able to produce stability results for almost representations in permutations, see the discussion in \cite{chapman2023stability}. 

In this article we study groups $\Gamma$, particular torsionfree lattices in the algebraic group ${\rm PSp}_{2d}(\mathbb Q_p)$ for $d\geq 4$ and $p$ large. These groups are residually finite and hence sofic, but admit finite central extensions $\tilde{\Gamma}$ which are not residually finite anymore. This phenomenon was first discovered by Deligne \cite{MR507760} in work on the congruence subgroup problem. Our main result, Theorem \ref{thm:main}, says that these central extensions $\tilde{\Gamma}$ cannot be sofic if the group $\Gamma$ is stable in a certain sense, see Section \ref{sec:stable}.

On a more conceptual level, which might be interesting in its own right, we introduce cohomological invariants that obstruct containment of p.m.p.\ actions of groups. These obstructions can be enhanced in the presence of cosystolic inequalities to obstruct also weak containment. This is exactly the route our proof takes: we show that for any sofic approximation of the finite central extension $\tilde \Gamma$, the induced sofic approximation of the lattice $\Gamma$ admits a limit action which is not weakly contained in finite actions of the lattice. This contradicts a very weak form of stability that we introduce and discuss in Section \ref{sec:stable}. We call a group \emph{stable in finite actions} if any partition of a suffiently good sofic approximation can be modelled (in the spirit of Kechris' notion of weak containment) in a finite action, see Definition \ref{def:stable}. This seems much weaker than any other notion of stability that has been studied so far.

All this is formulated in our main result, which is Theorem \ref{thm:main}, and should be compared with results of Bowen--Burton \cite{MR4105530}, see the discussion in Section \ref{sec:proof}. It can also be compared with a result of Dogon in the hyperlinear setting, see \cite{dogon}.

The paper is organized as follows. In Section \ref{sec:2}, we recall basics on measured Boolean algebras, weak containment of actions, metric group cohomology, Bruhat-Tits buildings and high-dimensional expanders. In Section \ref{sec:3}, we recall the definitions of soficity and various notions of stability. We introduce a new notion of stability, weaker than most notions that have been studied so far, and prove our main theorem. Finally, we review in Section \ref{Chapter:Candidate} the construction of a group satisfying all conditions in our main theorem.

\vspace{0.2cm}
The results of this article also appeared in the doctoral thesis of the first-named author, see \cite{thesisgohla}, with some text overlap.

\section{Group cohomology and applications} \label{sec:2}

\subsection{Measured Boolean algebras}

Since our setup requires taking metric ultraproducts and we would like to avoid working with non-standard measure spaces, we will quickly explain the setup of measured Boolean algebras and metric rings. This way, we can focus on the algebraic aspects of our reasoning. We will be rather brief in introducing metric ultraproducts for measured Boolean algebras and their associated abelian groups of $A$-valued functions for some finite abelian group $A$.

A measured Boolean algebra $({\rm P},\mu)$ is a Boolean algebra $({\rm P},\cap,\cup,\neg,0,1)$ with a measure $\mu \colon {\rm P} \to [0,1]$, i.e., a map such that
$$\mu(X \cup Y) + \mu(X \cap Y) = \mu(X) + \mu(Y) \quad \mbox{and} \quad \mu(0)=0, \ \mu(1)=1.$$
We say that $(\rm P,\mu)$ is \emph{complete} if ${\rm P}$ is complete with the metric $d(X,Y) = \mu(X \cup Y) - \mu(X \cap Y).$ The measure $\mu$ is said to be \emph{faithful}, if $\mu(X)=0$ only when $X=0$. We typically assume that $({\rm P},\mu)$ is complete and that $\mu$ is faithful, even though we will not need these conditions. A typical example for $({\rm P},\mu)$ to keep in mind is the measure algebra of measurable subsets of a standard probability space up to measure zero. In particular, this includes the power set of a finite set with the normalized counting measure.

Let $A$ be a finite abelian group. We write ${\rm P}(A)$ for the abelian group generated by symbols $X \otimes a$ for $X \in {\rm P}$ and $a \in A$, subject to the identity
$$X \otimes a + Y \otimes b = (X \setminus Y) \otimes a + (Y \setminus X) \otimes b + (X \cap Y) \otimes (a+b),$$
for all $X,Y \in {\rm P}$ and $a,b \in A.$
If $A$ is a finite commutative ring, then ${\rm P}(A)$ carries a natural multiplication which is defined via the formula
$(X \otimes a)(Y \otimes b) = (X \cap Y) \otimes (ab).$ It is easy to see that every element $x \in {\rm P}(A)$ can be written uniquely in the form 
$$x= \sum_{a \in A, a \neq 0} X_a \otimes a$$ with $X_a \cap X_b = \varnothing$ if $a\neq b.$
We denote the natural inclusion $A$ into ${\rm P}(A)$, given by $a \mapsto 1 \otimes a$, by $\theta^{\rm P} \colon A \to {\rm P}(A).$
Note that one can extend the measure $\mu$ to ${\rm P}(A)$ via the formula:
$$\mu(x) = \sum_{a \in A, a \neq 0} \mu(X_a) \in [0,1].$$
This turns ${\rm P}(A)$ into a metric abelian group via the formula $d(x,y)=\mu(x-y).$

Let $I$ be a set and $({\rm P}_j,\mu_j)_{j \in I}$ be a family of measured Boolean algebras indexed by $I$. Let $\mathcal U \in \beta I$ be a non-principal ultrafilter. We set
$${\rm P}_{\mathcal U} = \prod_{j \in I} {\rm P}_j / \sim$$
with $(X_j)_j \sim (Y_j)_j$ if and only if $\lim_{j \to {\mathcal U}} d_i(X_j,Y_j)=0.$ The equivalence class of $(X_j)_j$ is denoted by $[X_j]_{j}$. Note that ${\rm P}_{\mathcal U}$ is again a measured Boolean algebra with measure
$$\mu_{\mathcal U}([X_j]_j) = \lim_{j \to {\mathcal U}} \mu_j(X_j).$$ The measured Boolean algebra $({\rm P}_{\mathcal U},\mu_{\mathcal U})$ is called the metric ultraproduct. In a similar way, we can consider the metric ultraproduct $({\rm P}(A)_{\mathcal U},\mu_{\mathcal U})$ of metric abelian groups $({\rm P}_j(A),\mu)$ and note that ${\rm P}_{\mathcal U}(A)={\rm P}(A)_{\mathcal U}$ in a natural way.

Let $\Gamma$ be a countable discrete group. We consider measured Boolean algebras with a measure preserving $\Gamma$-action and in this case we speak of a measured $\Gamma$-Boolean algebra. For a p.m.p.\ preserving $\Gamma$-action $(X,\mu)$, we denote by ${\rm M}(X,\mu)$ its measure algebra, which is naturally a measured $\Gamma$-Boolean algebra. If $A$ is a finite abelian group, then ${\rm P}(A)$ is naturally a $\mathbb Z[\Gamma]$-module whenever ${\rm P}$ is a measured $\Gamma$-Boolean algebra. In case of finite measured $\Gamma$-Boolean algebras, we simply speak of \emph{finite actions}. We denote the family of finite transitive actions by $\mathcal P_{\rm tra}$ and the family of finite actions by $\mathcal P_{\rm fin}.$

\begin{definition}
Let $({\rm P},\mu)$ and $({\rm Q},\nu)$ be measured $\Gamma$-Boolean algebras with measure preserving $\Gamma$-actions. We say that $({\rm P},\mu)$ is contained in $({\rm Q},\nu)$ if there exists a measure-preserving and $\Gamma$-equivariant homomorphism of Boolean algebras from $({\rm P},\mu)$ to $({\rm Q},\nu)$.
\end{definition}

Clearly, any metric ultraproduct of measured $\Gamma$-Boolean algebras inherits a natural measure preserving $\Gamma$-action from the factors. We sometimes simply speak of the \emph{limit action} if the context is clear. Using the notion of metric ultraproduct we can introduce the fundamental notion of \emph{weak containment}.

\begin{definition}
We say that a measured $\Gamma$-Boolean algebra $({\rm P},\mu)$ is \emph{weakly contained} in a family of measured $\Gamma$-Boolean algebras if $({\rm P},\mu)$ is contained in a metric ultraproduct of measured $\Gamma$-Boolean algebras from that family.
\end{definition}

This definition of weak containment corresponds to the notion of weak containment of measure preserving actions on probability spaces, see \cites{kechris, MR4138908}.

\subsection{Metric group cohomology} \label{sec:coho}

We refer to Brown \cite{MR1324339} for the basics on group cohomology. Let us just recall some basic notation. Let $B\Gamma$ be a classifying space for $\Gamma$, i.e., the $\Gamma$-quotient of a contractible simplicial complex $E\Gamma$ with a free and simplicial $\Gamma$-action. Let $M$ be a $\mathbb Z\Gamma$-module. We denote by ${\rm C}^*(B\Gamma,M) \coloneqq \hom_{\mathbb Z \Gamma}({\rm C}_*(E \Gamma),M)$ the cochain complex with differentials
$$\delta_i \colon {\rm C}^i(B\Gamma,M) \to {\rm C}^{i+1}(B\Gamma,M).$$
We set as usual $${\rm Z}^i(B\Gamma,M) = \ker \delta_i, \quad {\rm B}^i(B\Gamma,M) = {\rm im} \ \! \delta_{i-1}$$ and define the group cohomology with coefficients in $M$ as
$${\rm H}^i(B\Gamma,M) = \frac{{\rm Z}^i(B\Gamma,M)}{{\rm B}^i(B\Gamma,M)}.$$

Let ${\rm P}$ be a measured $\Gamma$-Boolean algebra, $A$ be a finite abelian group and $n \in \mathbb N$. Then, there exists a natural homomorphism
$$\theta^{\rm P}_* \colon {\rm H}^i(B\Gamma,A) \to {\rm H}^i(B\Gamma,{\rm P}(A)),$$
where ${\rm P}(A)$ and $\theta^{\rm P}$ as in the previous section. 
Let $\Gamma$ be of finite type and $B\Gamma$ a finite model for its classifying space and let $B\Gamma(i)$ denote the set of $i$-simplices. Then, fixing a faithful probability measure $\nu_i$ on $B\Gamma(i)$, we can endow ${\rm C}^i(B\Gamma,{\rm P}(A))$ with a natural metric that measures non-triviality of the co-chain, i.e. we define a length function by the formula
$$\|c\|_{\mu,\nu}\coloneqq \sum_{x \in B\Gamma(i)}  \mu(c(x)) \nu_i(x),$$
where we view $c \in {\rm C}^i(B\Gamma,{\rm P}(A)) = \hom_{\mathbb Z \Gamma}({\rm C}_i(E \Gamma),{\rm P}(A))$ simply as a function $c \colon B\Gamma(i) \to {\rm P}(A)$ after choosing a section of the map $E\Gamma(i) \to B\Gamma(i)$. Note that the length is independent of this choice by $\Gamma$-invariance of the $\mu.$
Note also that $\delta_i \colon {\rm C}^i(B\Gamma,M) \to {\rm C}^{i+1}(B\Gamma,M)$ is Lipschitz with respect to the length functions introduced above with a Lipschitz constant just depending on the cell structure of $B\Gamma$ and the measures $\nu_i$ and $\nu_{i+1}$.

There is a natural quotient pseudo-metric on ${\rm H}^i(B\Gamma,{\rm P}(A))$ which we also denote by $$\vertiii{.}_{\mu,\nu} \colon {\rm H}^i(B\Gamma,{\rm P}(A)) \to [0,1].$$
Note that this pseudo-metric depends on the choice of finite model for $B \Gamma$ and the measure $\nu_i$, but it is unique up to bi-Lipschitz equivalence of pseudo-metrics. (In order to avoid the ambiguity, we fix a finite model for $B\Gamma$ and choices for the measures $\nu_i$ for the rest of the text and suppress the dependence on the $\nu_i$'s in the notation.)
In particular 
$${\rm NH}^i(B\Gamma,{\rm P}(A)) := \{ \alpha \in {\rm H}^i(B\Gamma,{\rm P}(A)) \mid \vertiii{\alpha} = 0 \}$$
is well-defined and independent of a choice of a finite model. We also define the reduced cohomology
$${\rm H}_{\rm red}^i(B\Gamma,{\rm P}(A)) = \frac{{\rm H}^i(B\Gamma,{\rm P}(A))}{{\rm NH}^i(B\Gamma,{\rm P}(A))}.$$

There are two natural choices for the measures on $B\Gamma(i)$
that we will briefly discuss.

\begin{definition}
\label{ExNormOnCW2}
Let $X$ be a finite simplicial complex of dimension $d$ and $X(i)$ be the set of $i$-simplices. We define
$$\mu(\sigma) \coloneqq \frac{|\{ \tau \in X(d) \mid \sigma \subset \tau \}|}{\binom{d+1}{i+1} \cdot |X(d)|}.$$
\end{definition}

This is the most common choice in the literature. Note however, that this particular choice has, a priori, the problem of $i$-simplices which are ignored in case an $i$-simplex is not contained in any $d$-dimensional cell. In the sequel, we will always assume that all simplicial complexes are finite dimensional and \emph{pure}, i.e. every $i$-simplex is contained in at least one simplex of maximal dimension. This is true for buildings, which are the only relevant examples which we discuss.

The other common choice for the weights  differs just by a normalization factor which depends on the dimension.
\begin{definition}
\label{ExNormOnCW3}
Let $X$ be a finite simplicial complex of dimension $d$ and $X(i)$ be the set of $i$-simplices. We define
$$m(\sigma ) \coloneqq (d-i)! \cdot |\{ \tau \in X(d) \mid \sigma \subseteq \tau \}|.$$
\end{definition}
This choice comes with the additional property that the weight of a $k$-dimensional simplex is given as the sum of all weights of the $(k+1)$-dimensional simplices it is contained in.

\subsection{Shapiro's lemma in the metric context}

In this section we want to tie together the metric and homological aspects of co-induction. Let $B\Gamma$ be a finite classifying space for $\Gamma$ and $\Lambda \leq \Gamma$ be a subgroup of finite index. A natural model for $B\Lambda$ is the finite covering $E\Gamma/\Lambda$ and we will fix this model for $B\Lambda$. Given any probability measure on $B\Gamma(i)$, there is a natural probability measure on $B\Lambda(n)$, which arises as the normalized equidistribution on the fibers of the covering map $B\Lambda \to B\Gamma.$ The well-known Shapiro's lemma says that the cohomology of $\Lambda$ with coefficients in a $\mathbb Z\Lambda$-module $M$ can be computed in terms of the cohomology of $\Gamma$ with respect to the co-induced module $\hom_{\Lambda}(\mathbb Z\Gamma,M).$ If $M=A$ is just an abelian group with trivial $\Lambda$-action, then $\hom_{\Lambda}(\mathbb Z\Gamma,M) = {\rm map}(\Gamma/\Lambda,A) = {\rm P}^{\Gamma/\Lambda}(A),$ where ${\rm P}^{\Gamma/\Lambda}$ denotes the Boolean algebra of subset of $\Gamma/\Lambda$. The following lemma states that the Shapiro isomorphism is in fact isometric when we consider ${\rm P}^{\Gamma/\Lambda}$ with the normalized counting measure and with the choices of measures on cells as discussed above.

\begin{lemma}[Shapiro] \label{lem:shapiro}
Let $\Lambda \leq \Gamma$ be a subgroup. There is a natural isometric isomorphism
$${\rm H}^i(B\Lambda,A) \stackrel{\kappa}{\to} {\rm H}^i(B\Gamma,{\rm P}^{\Gamma/\Lambda}(A))$$
such that $$\kappa([c]|_{\Lambda}) = \theta^{\Gamma/\Lambda}_*([c]) \in {\rm H}^i(B\Gamma,{\rm P}^{\Gamma/\Lambda}(A))$$
for all $[c] \in {\rm H}^i(B\Gamma,A)$.
\end{lemma}

This lemma has particularly interesting consequences when there are infinitely many subgroups of finite-index and the metric properties of cohomology groups and groups of cocycles and coboundaries can be controlled uniformly. 

\subsection{Bruhat-Tits buildings}

Let us continue by recalling a particular family of classifying spaces. For background on linear algebraic groups and basic notions we refer to \cite{borel}. The standard reference for the theory of buildings is \cite{AB}.

\begin{definition}
Let $G$ be a reductive algebraic group over $\mathbb{Q}_{p}$. Then there exists a simplicial complex $\mathcal{B}(\mathbf{G})$ such that $\mathbf{G}(\mathbb Q_p)$ acts continuously on $\mathcal{B}(\mathbf{G})$ by simplicial automorphisms and
\begin{enumerate}[$(i)$]
\item the action of $\mathbf{G}(\mathbb Q_p)$ is proper,
\item $\mathcal{B}(\mathbf{G})$ is contractible,
\item $\mathcal{B}(\mathbf{G})$ is of dimension ${\rm rk}_{\mathbb{Q}_{p}} G <\infty$ and
\item $\mathcal{B}(\mathbf{G})$ is locally finite, i.e. all proper links are finite.
\end{enumerate}
$\mathcal{B}(\mathbf{G})$ will be called the \emph{Bruhat-Tits building of $\mathbf{G}$}. \label{DefBTbuilding}
\end{definition}

The following lemma is well-known.

\begin{lemma} \label{lem:bruhat}
Let $\Gamma$ be a torsion-free lattice of $\mathbf{G}(\mathbb Q_p)$. Then $\mathcal{B} (\mathbf{G})/\Gamma$ is a finite classifying space of $\Gamma$. \label{TheoBTisClassSpace}
\end{lemma}

Bruhat-Tits buildings are sources of all kind of rigidity phenomena as will be discussed in the next section.

\subsection{Cosystolic inequalities and expansion}
\label{sec:highex}

In this section, we follow the exposition of Lubotzky \cite{MR3966743}, see also \cites{MR3536553, kaufman2018cosystolic, LMM}. It is natural to wonder in what sense the metric and the algebraic structure on the cohomology groups discussed in Section \ref{sec:coho} are compatible. One immediate requirement is that we would want that ${\rm H}^i(B\Gamma,{\rm P}(A))$ is Hausdorff, i.e. the pseudo-metric is in fact a metric. However, it turns out that it is of even greater importance to require that the induced metric is discrete. Another subtle point is the comparison of the quotient pseudo-metric and the subspace metric on the subgroup $${\rm im}\ \delta_{i-1} \subset {\rm C}^i(B\Gamma,{\rm P}(A)).$$

\begin{definition} \label{hexdef} Let $\varepsilon,\mu>0$ and let $\Gamma$ be a group of finite type and fix a finite model for $B\Gamma$. Let $\mathcal P$ be a family of measured $\Gamma$-Boolean algebras  and $A$ be a finite abelian group. We say that $\Gamma$ satisfies
\begin{enumerate}[$(i)$]
\item \label{EqCosystolicIneq}
a \emph{cosystolic inequality} in dimension $i$ with respect to $\mathcal P$ and $A$, if there exists $\varepsilon>0$, such that
\begin{equation*} 
\vertiii{\alpha}_i \geq \varepsilon
\end{equation*}
for all non-zero $\alpha \in {\rm H}^{i}(B\Gamma,{\rm P}(A))$ and all $({\rm P},\mu) \in \mathcal P,$ and
\item \label{EqExpanderIneq}
\emph{expansion} in dimension $i$ with respect to $\mathcal P$ and $A$ if there exists $\varepsilon>0$, such that
\begin{equation*}
\| \delta_i (c)\|_{i+1} \geq \varepsilon \inf_{z\in {\rm Z}^{i}(B\Gamma,{\rm P}(A))} \| c+z\|_{i}, 
\end{equation*}
for all $c \in {\rm C}^{i}(B\Gamma,{\rm P}(A)) $ and all ${\rm P} \in \mathcal P.$
\end{enumerate}
\end{definition}

Expansion is visible only in the inner workings of the co-chain complex and says that the quotient pseudo-metric and the subspace metric are Lipschitz equivalent on ${\rm im} \ \! \delta_i$. Note that by the discussion in Section \ref{sec:coho} these properties are independent of the choice of the finite model $B\Gamma$, so in fact they are properties of $\Gamma.$

\begin{definition}
We say that a $d$-dimensional finite complex with fundamental group $\Gamma$ is a cosystolic expander with respect to $\mathcal P$ and $A$, if it satisfies a cosystolic inequality and expansion in dimensions $0 \leq i \leq d-1.$
Moreover, we call a cosystolic expander a coboundary expander, if in addition its cohomology vanishes in dimensions $1 \leq i \leq d-1$
\end{definition}

A typical result says that the $(d-1)$-skeleton of a finite quotient of a $d$-dimensional Bruhat-Tits building is (under certain conditions) a cosystolic expander in a uniform way, i.e., we need that $\mu,\varepsilon$ above do not depend on the particular quotient. We conclude that the fundamental group of such a quotient satisfies cosystolic inequalities and expansion in dimensions $0 \leq i \leq d-2$ with respect to the family of finite actions. More precisely, we will obtain the following result:

\begin{theorem} \label{CorPSpExpansion} Let $d \geq 4$ and $p$ large enough.
Let $\Gamma$ be a torsionfree lattice in ${\rm PSp}_{2d} (\mathbb{Q}_{p})$ and let $A$ be a finite abelian group. Then, $\Gamma$ satisfies a cosystolic inequality and expansion with respect to the family of finite transitive $\Gamma$-actions and $A$ in dimensions $0 \leq i \leq d-2.$
\end{theorem}

The goal of this section is to prove the following theorem, which implies the previous theorem by the discussion above and Lemma \ref{lem:shapiro} -- it is itself a generalization of \cite{MR3966743}*{Thm. 3.8} to arbitrary abelian groups.

\begin{theorem}
\label{TheoExpansionOverAnyGroup}
For $2\leq d \in \mathbb{N}$ there exist $\varepsilon = \varepsilon(d) > 0$ and $q_{0} = q_{0} (d)$ such that if ${\mathbb K}$ is a local non-archimedean field of fixed residue degree $q > q_{0}$, $\mathbf{G}$ a simple ${\mathbb K}$-group of ${\mathbb K}$-rank $d$. Let $A$ be a finite abelian group and let $\Lambda$ a torsion-free lattice of $\mathbf{G}(\mathbb{K})$ that acts with injectivity radius at least $3$ on the Bruhat-Tits building $\mathcal{B} (\mathbf{G})$.

Then, the $(d-1)$-skeleton $Y$ of the finite quotient $\mathcal{B}(\mathbf{G}) /\Lambda$ is a $(d-1)$-dimensional $(\varepsilon ,\varepsilon )$-cosystolic expander over $A$.
\end{theorem}
Let's first spell out how Theorem \ref{CorPSpExpansion} can be derived.
\begin{proof}[Proof of Theorem \ref{CorPSpExpansion}]
Let $\Gamma$ be as above and $\Gamma \curvearrowright \Gamma/\Lambda$ be a finite transitive action. Then the cochain complex computing ${\rm H}^i(\Gamma,{\rm P}^{\Gamma/\Lambda}(A))$ is isometrically isomorphic to the the simplical cochain complex ${\rm C}^*(\mathcal{B}(\mathbf{G}) /\Lambda,A)$. Thus, if the $(d-1)$-skeleton of $\mathcal{B}(\mathbf{G}) /\Lambda$ satisfies cosystolic and expansion inequalities this translates into the corresponding inequalities for group cohomology with coefficients in ${\rm P}^{\Gamma/\Lambda}(A)$. Since the estimates of Theorem \ref{TheoExpansionOverAnyGroup} are independent of $\Lambda$, we obtain uniform estimates for all finite transitive actions.
\end{proof}

The result is essentially contained in the literature, even though it took us some effort to find the right results and understand how to combine them. We will prove Theorem \ref{TheoExpansionOverAnyGroup} as a consequence of the following result of Kaufman--Mass.

\begin{theorem}[Theorem 3.3 and 3.4, \cite{kaufman2018cosystolic}]
Let $d, Q \in \mathbb{N}$ and $X$ be a finite $d$-dimensional simplicial complex of $Q$-bounded degree, $A$ a finite abelian group and $\beta \in (0,1)$. There exists a $\mu = \mu (d,\beta ) >0$ such that if
\begin{enumerate}[$(i)$]
\item for every $\emptyset \neq \sigma \in X$ the link $X_{\sigma}$ is a $\beta$-coboundary expander over $A$ and
\item for every $\sigma \in X$ the link $X_{\sigma }$ is a $\mu$-skeleton expander,
\end{enumerate}
then the $(d-1)$-skeleton of $X$ is a $(\min \{ Q^{-2} ,\mu \} ,\mu)$-cosystolic expander over $A$. \label{TheoGeneralK-M}
\end{theorem}

\begin{proof}[Proof of Theorem \ref{TheoExpansionOverAnyGroup}]
We need to check that the conditions of Theorem \ref{TheoGeneralK-M} are satisfied and the bounds are independent of $\Gamma$. First of all, note that $X$ is of $Q$-bounded degree independent of the choice of $\Gamma$.
By \cite{AB}*{Proposition 4.9} every link of a building is still a building and we already know that every proper link is finite.  By \cite{kaufman2018cosystolic}*{Theorem 4.1} every finite building is a coboundary expander over any abelian group. This shows that every proper link $X_{\sigma}$ is a $\beta_{d-|\sigma |}$-coboundary expander over any abelian group. Hence, Condition $(i)$ is satisfied.

The verification of Condition $(ii)$ is split into two parts. First we consider proper links and then the remaining degenerate case.

Every reductive algebraic group $\mathbf{G}$ admits a BN-pair and by \cite{AB}*{Theorem 6.56} $\mathbf{G}(\mathbb K)$ acts strongly transitively on its Bruhat-Tits building. Now, \cite{MR3536553}*{Lemma 5.17} states that a suitable stabilizer also acts strongly transitively on all links of said building. 
Hence \cite{MR3536553}*{Lemma 5.14} states that all proper links are regular complexes and therefore they are $\lambda (X_{\sigma})$-skeleton expanders by \cite{MR3536553}*{Corollary 4.7}. Since there are only finitely many isomorphisms classes of links in $\mathcal{B}(\mathbf{G})$, this allows us to to find a bound on the skeleton expansion of proper links in $X$ independently of the specific $\Gamma$ we chose.

It remains to show skeleton expansion for $X_{\emptyset } =X$ and here the argument gets slightly more complicated. Again, the proof relies on a local-global method. As we have mentioned before every proper link in $X$ appears as a proper link in $\mathcal{B}(\mathbf{G})$. Furthermore we know that every link of $\mathcal{B}(\mathbf{G})$ is again a building. This means two things for us:

\begin{enumerate}[$(a)$]
\item Every link of $X$ is connected.
\item By \cite{kaufman2018cosystolic}*{Theorem 4.1} spherical buildings are coboundary expanders. In particular, this shows that the $1$-dimensional links of $X$ are expander graphs with a Cheeger constant which is bounded by a constant that is independent of the specific link.
\end{enumerate}
Therefore $X$ is a \emph{$\lambda$-local spectral expander} in the notation of \cite{MR3755725} and the assumptions of \cite{MR3755725}*{Corollary 6.1} are satisfied. In particular, then the non-trivial spectrum of the upper Laplacian $\Delta^{+} \coloneqq \Delta^{+}_{0}$ with respect to the inner product $\langle f,g\rangle \coloneqq \sum_{\sigma \in \Sigma (k)} \frac{m(\sigma )}{k!} \cdot f(\sigma )\cdot g(\sigma )$ on ${\rm C}^{k} (X,\mathbb{R} )$ and with $m$ from Remark \ref{ExNormOnCW3} is contained in $[a_{\alpha} ,\infty )$ with $a_{\alpha} >0$ and $\lim_{\alpha \rightarrow k-1} a_{\alpha} =1$. This implies that $X$ is a $\mu$-skeleton expander as required.
\end{proof}

\subsection{An application to weak containment}
Let $\kappa \colon {\rm P} \to {\rm Q}$ be a measure-preserving, $\Gamma$-equivariant homomorphism of Boolean algebras. Let $A$ be a finite abelian group. Note that the induced map $\kappa_* \colon {\rm H}^i(B\Gamma,{\rm P}(A)) \to {\rm H}^i(B\Gamma,{\rm Q}(A))$ is contractive. This simple observation allows us to find a cohomological obstruction to weak containment of a measured $\Gamma$-Boolean algebra provided a suitable cosystolic inequality holds. This is explained in the next lemma and its corollary.

\begin{lemma} \label{lem:limit}
Let  $\Gamma$ be  of finite type, $\mathcal P$ be a family of measured $\Gamma$-Boolean algebras and $A$ be an abelian group. Let $({\rm P}_j,\mu_j)_{j \in I}$ be an indexed set of elements of $\mathcal P$ and ${\mathcal U} \in \beta I$ be a non-principal ultrafilter. Then, the natural map
$${\rm H}_{\rm red}^i(\Gamma,{\rm P}_j(A))_{\mathcal U} \to {\rm H}_{\rm red}^i(\Gamma,{\rm P}_{\mathcal U}(A))$$
is an isometric injection and it is surjective if $\Gamma$ satisfies expansion in dimension $i$ with respect to the family $\mathcal P$ and $A.$
In particular, we have
$$\vertiii{\theta^{{\rm P}_{\mathcal U}}_*(\alpha)}_{\mu_{\mathcal U}} = \lim_{j \to {\mathcal U}} \vertiii{\theta^{{\rm P}_j}_*(\alpha)}_{\mu_j}$$
for all $\alpha \in {\rm H}^i(\Gamma,A)$.
\end{lemma}
\begin{proof} Let $\alpha_j \in {\rm H}^i(\Gamma,{\rm P}_j(A))$
and $\delta>0$ for $j \in I$. Denote the image of in ${\rm H}^i(\Gamma,{\rm P}_{\mathcal U}(A))$ by $\alpha_{\mathcal U}$.
The inequality $$\vertiii{\alpha_{\mathcal U}}_{\mu_{\mathcal U}} \leq \lim_{j \to {\mathcal U}} \vertiii{\alpha_j}_{\mu_j}$$ is obvious. Indeed, let $\delta>0$ be arbitrary and let $\beta_j \colon B\Gamma(i) \to {\rm P}_j(A)$ be cocycle representatives of $\alpha_j$ such that $\vertiii{\alpha_j }_{\mu_{j}} \geq \|\beta\|_{j} - \delta.$ Then, $\beta = [\beta_j]_j$ is by definition a cocycle representative of $\alpha_{\mathcal U}$ and
$$\vertiii{\alpha_{\mathcal U} }_{\mu_{\mathcal U}} \leq \|\beta\|_{\mu_{\mathcal U}} = \lim_{j \to \mathcal U} \|\beta_j\|_{\mu_j} \leq \lim_{j \to \mathcal U} \vertiii{\alpha_{j} }_{\mu_{j}} + \delta.$$ Since $\delta>0$ was arbitrary, this proves the first inequality.

Let now $\beta' \colon B\Gamma(i) \to {\rm P}_{\mathcal U}(A)$ be a cocycle representative of $\alpha_{\mathcal U}$, such that
$\vertiii{\alpha_{\mathcal U} }_{\mu_{\mathcal U}} \geq \|\beta'\|_{\mu_{\mathcal U}} - \delta.$
Note that $\delta_i(\beta')=0$ and $\beta' - \beta= \delta_{i-1}(\zeta)$ for some $\zeta = [\zeta_j]_j \colon B\Gamma(i-1) \to {\rm P}_{\mathcal U}(A).$ We set $\beta'_j := \beta_j - \delta_{i-1}(\zeta_j)$ and note that $[\beta'_j]_j = \beta'$. The elements $\beta'_j$ are clearly cocycles that represent $\alpha_j$. Hence
$$\lim_{j \to \mathcal U} \vertiii{\alpha_{j} }_{\mu_{j}} \leq \lim_{j \to \mathcal U} \|\beta_j' \|_{\mu_{j}} = \|\beta'_{\mathcal U} \|_{\mu_{\mathcal U}} \leq \vertiii{\alpha_{\mathcal U} } + \delta.$$ This proves the other inequality and shows that the natural map is indeed an isometric embedding.

It remains to show that the map is surjective in presence of expansion. Let $\alpha_{\mathcal U} \in {\rm H}^i(\Gamma,{\rm P}_{\mathcal U}(A))$ be represented by $\beta \colon B\Gamma(i) \to {\rm P}_{\mathcal U}(A).$
Further, choose functions $\beta_j \colon B\Gamma(i) \to P_j(A)$, such that $[\beta_j]_j=\beta$. From \eqref{EqExpanderIneq} in Definition \ref{hexdef}, we conclude that there exists $z_j \in {\rm Z}^i(B\Gamma,P_j(A))$ with
$$\lim_{j \to {\mathcal U}}\|\delta_{i}(\beta_j)\| \geq \varepsilon \lim_{j \to {\mathcal U}} \|\beta_j - z_j\|.$$

We obtain
$$0 = \|\delta_i(\beta)\|_{\mu_{\mathcal U}} = \lim_{j \to {\mathcal U}} \|\delta_i(\beta_j)\|_{\mu_j} \geq \varepsilon\lim_{j \to {\mathcal U}} \|\beta_j-z_j\|_{\mu_j}.$$
Since the image of $[z_j]_j$ in ${\rm H}^i(\Gamma,{\rm P}_{\mathcal U}(A))$ is exactly $\alpha_{\mathcal U}$, we conclude that surjectivity of the map as claimed.
\end{proof}


\begin{corollary} 
\label{cor:weakcont}
Let $\Gamma$ be a group of finite type, let $({\rm P},\mu)$ be a measured $\Gamma$-Boolean algebra and let $\mathcal P$ be a family of measured $\Gamma$-Boolean algebra. If 
$$\vertiii{\theta^{\rm P}_*(\alpha)}_{\mu} < \inf \{ \vertiii{\theta^{{\rm Q}}_*(\alpha)}_{\nu} \mid ({\rm Q},\nu) \in \mathcal P \}$$
for some $\alpha \in {\rm H}^i(\Gamma,A),$
then $({\rm P},\mu)$ is not weakly contained in the family $\mathcal P$.
\end{corollary}
\begin{proof} This a consequence of Lemma \ref{lem:limit} which is applicable because of Condition $(i)$. Indeed, if $({\rm P},\mu)$ is contained in some ultra-product $({\rm P}_{\mathcal U},\mu_{\mathcal U})$ for some ultrafilter ${\mathcal U} \in \beta I$ and $({\rm P}_j,\mu_j) \in \mathcal P$ for $j \in I$, then the natural map
$${\rm H}^i(\Gamma,{\rm P}(A)) \to {\rm H}^i(\Gamma,{\rm P}_{\mathcal U}(A))$$ is contractive and hence
$$\vertiii{\theta^{\rm P}_*(\alpha)}_{\mu} \geq \vertiii{\theta^{{\rm P}_{\mathcal U}}_*(\alpha)}_{\mu_{\mathcal U}} = \lim_{j \to {\mathcal U}} \vertiii{\theta^{{\rm P}_j}_*(\alpha)}_{\mu_j} \geq \inf \{ \vertiii{\theta^{{\rm Q}}_*(\alpha)}_{\nu} \mid ({\rm Q},\nu) \in \mathcal P \}.$$
However, this contradicts our assumption. 
\end{proof}

\subsection{An application to residually finite groups}

Let $\Gamma$ be a residually finite group. Let $A$ be a finite abelian group and let $\alpha \in {\rm Z}^2(\Gamma,A)$ be a 2-cocycle that classifies a central extension
$$1 \to A \to \tilde \Gamma \to \Gamma \to 1.$$

The following lemma provides a cohomological consequence of the situation that $\tilde\Gamma$ is not residually finite. 

\begin{lemma} \label{lem:nonzero}
Let $\Gamma$ be a residually finite group, let $\Lambda \leq \Gamma$ be a subgroup of finite index and let $A$ be a finite abelian group. If $\alpha \in {\rm H}^2(\Gamma,A)$ classifies a central extension $\tilde \Gamma$ which is not residually finite, then 
$$0 \neq \theta^{\Gamma/\Lambda}_*(\alpha) \in {\rm H}^i(\Gamma,{\rm P}^{\Gamma/\Lambda}(A)).$$
\end{lemma}
\begin{proof}
By Lemma \ref{lem:shapiro}, we have $\kappa(\alpha|_{\Lambda}) = \theta^{\Gamma/\Lambda}_*(\alpha)$, so that $\alpha|_{\Lambda}$ must vanish if $\theta^{\Gamma/\Lambda}_*(\alpha)=0.$ But then, the inverse image of $\Lambda$ in $\tilde \Gamma$ is isomorphic to $\Lambda \times A$. In particular, $\tilde \Gamma$ contains a residually finite subgroup of finite index. This implies that $\tilde \Gamma$ is also residually finite in contradiction to our assumption.
\end{proof}

The preceding lemma has an interesting consequence in presence of a suitable cosystolic inequality in dimension $2$. We denote the profinite completion of a residually finite group $\Gamma$ by $\hat \Gamma.$ Recall that it carries a natural Haar measure.

\begin{theorem} \label{thm:neqzero}
Let $A$ be a finite abelian group.
Let $\Gamma$ be a residually finite group which satisfies a co-systolic inequality in dimension $2$ for all finite transitive actions with respect to $A$. If $\alpha \in {\rm H}^2(\Gamma,A)$ classifies a central extension $\tilde \Gamma$ which is not residually finite, then
$$0 \neq \theta^{\rm P}_*(\alpha) \in {\rm H}^2(\Gamma,{\rm P}(A))$$
whenever the measured $\Gamma$-Boolean algebra ${\rm P}$ is weakly contained in the family of finite actions. In particular, this holds for ${\rm P}={\rm M}(\hat \Gamma).$
\end{theorem}
\begin{proof} This is a consequence of Lemma \ref{lem:nonzero} and Corollary \ref{cor:weakcont}. Indeed, by the cosystolic inequality, there is a uniform lower bound for the length of the cohomology class $\theta^{\rm P}_*(\alpha) \in {\rm H}^2(\Gamma,{\rm P}(A))$ for all finite transitive actions. Since any finite action is a disjoint union of transitive actions, this implies the same lower bound on the length of $\theta^{\rm P}_*(\alpha)$ for \emph{all} finite actions. Now, Corollary \ref{cor:weakcont} yields a lower bound on the length of $\theta^{\rm P}_*(\alpha)$.
\end{proof}

\section{Approximation and stability} \label{sec:3}

\subsection{Sofic approximations}

For background on sofic groups and sofic approximations we refer to the excellent survey by Pestov \cite{pestov}.
Let $\Gamma$ be a finitely presented group. Fix a surjection $\pi \colon {\rm F} \to \Gamma$, where ${\rm F}$ is a finitely generated free group and let $R \subseteq \ker(\pi)$ be a finite set that generates $\ker(\pi)$ as a normal subgroup. We set 
$$\ell_n(\sigma) = \frac1n |\{i \mid \sigma(i) \neq i\}|, \quad \forall \sigma \in {\rm Sym}(n)$$
and define the normalized Hamming metric as $d_n(\sigma,\tau)=\ell_n(\sigma\tau^{-1})$ for all $\sigma,\tau \in {\rm Sym}(n).$

\begin{definition}
A sofic representation of the group $\Gamma$ is a sequence of homomorphisms $\sigma_n \colon {\rm F} \to {\rm Sym}(k_n)$ such that
$\ell_{k_n}(r) \to 0$ for all $r \in R$ and $\ell_{k_n}(g) \to 1$ for all $g \notin \ker(\pi).$
\end{definition}

Note that each $\sigma_n$ gives rise to a measured ${\rm F}$-Boolean algebra ${\rm P}_n$. Moreover, after fixing an ultrafilter ${\mathcal U} \in \beta \mathbb N$, each sofic approximation gives rise to a natural \emph{limit action} given by the metric ultraproduct of measured ${\rm F}$-Boolean algebras.
Note that the limit action of a sofic approximation of $\Gamma$ is naturally a measured $\Gamma$-Boolean algebra.

\begin{lemma}
Let $\tilde \Gamma$ be a group with a finite normal subgroup $A$ and let $\Gamma \coloneqq \tilde \Gamma/A$ be the quotient. Every sofic approximation $(\sigma_n \colon {\rm F} \to {\rm Sym}(k_n))_n$ of $\tilde\Gamma$ is equivalent to a sofic approximation of the form $(\tilde \sigma_n \colon {\rm F} \to {\rm Sym}(m_n \times A))_n$, where $A$ acts regularly on the right factor. In particular, it induces a sofic approximation $(\sigma'_n \colon {\rm F} \to {\rm Sym}(m_n))_n$ of the group $\Gamma$.
\end{lemma}

Using the previous lemma makes it easy to study the corresponding limit actions.

\begin{lemma} \label{lem:inv}
Let ${\mathcal U} \in \beta \mathbb N$ be a non-principal ultrafilter and let ${\rm Q}_p$ be the limit $\tilde \Gamma$-action of a sofic approximation $(\sigma_n \colon {\rm F} \to {\rm Sym}(k_n))_n$ of $\tilde \Gamma$ as above. Then the limit $\Gamma$-action ${\rm P}_{\mathcal U}$ of the sofic approximation $(\sigma'_n \colon {\rm F} \to {\rm Sym}(m_n))_n$ of $\Gamma$ as above is given by the subalgebra of $A$-invariants, i.e., ${\rm P}_{\mathcal U} = ({\rm Q}_p)^A.$ 
\end{lemma}

Thus, the limit action of a sofic approximation of $\Gamma$, that is induced from a sofic approximation of $\tilde \Gamma$, is of a particular kind. Indeed, it arises as the $A$-invariants of a measured $\tilde \Gamma$-Boolean algebra, and we will make crucial use of this fact. We will now restrict to central extensions and show that for any such action, the cohomological classes obtained by pushing forward the classifying cocycle of the central extension vanish.

\begin{theorem} \label{thm:equalzero}
Let $\tilde{\Gamma}$ be a central extension of $\Gamma$ by the finite abelian group $A$ and let $\alpha \in {\rm H}^2(\Gamma,A)$ be the cohomology class that classifies this extension. Let $({\rm Q},\mu)$ be an $A$-free measured $\tilde \Gamma$-Boolean algebra and let ${\rm P}\coloneqq{\rm Q}^A$ be the subalgebra of $A$-invariants. Then, $$0=\theta^{\rm P}_*(\alpha) \in {\rm {\rm H}^2}(\Gamma, {\rm P}(A)).$$
\end{theorem}
\begin{proof} Consider the natural transformation of the low degree part of the Lyndon-Hochschild-Serre spectral sequence:
\begin{equation*}
\begin{tikzcd}
 {\rm H}^1(A,A)^{\Gamma} \ar[r] \ar[d] & {\rm H}^2(\Gamma,A) \ar[r] \ar[d,"\theta^P_*"] & {\rm H}^2(\tilde \Gamma,A) \ar[d] \\
 {\rm H}^1(A,{\rm Q}(A))^{\Gamma}  \ar[r] & {\rm H}^2(\Gamma,{\rm P}(A)) \ar[r] & {\rm H}^2(\tilde \Gamma,{\rm Q}(A))
\end{tikzcd}
\end{equation*} 
with ${\rm P}={\rm Q}^A$ by definition. Note that $\alpha \in {\rm H}^2(\Gamma,A)$ maps to zero in ${\rm H}^2(\tilde \Gamma,A)$. Hence, $\theta^{\rm P}_*(\alpha)$ maps to zero in ${\rm H}^2(\tilde\Gamma,{\rm Q}(A))$. The claim follows if we can show that ${\rm H}^1(A,{\rm Q}(A))^{\Gamma}=0.$ However, ${\rm Q}(A) = {\rm hom}_{\mathbb Z}(\mathbb Z[A],{\rm P}(A))$ as $\mathbb Z[A]$-module, since ${\rm Q}$ is assumed to be $A$-free. Hence, ${\rm H}^1(A,{\rm Q}(A)) =0$ and this finishes the proof. \end{proof}

\begin{remark} The vanishing of $\theta^{\rm P}_*(\alpha)$ for the limit action of a sofic approximation of $\Gamma$ which is induced by a sofic approximation $\tilde{\Gamma}$ can be proved directly along the approximation. We provide a sketch of the proof here: Let $S$ be a finite generating set of $\Gamma$ and let ${\rm F}_S$ be the free group on $S$. Pick a set-theoretic lift $\rho \colon \Gamma \to \tilde{\Gamma}$ and consider $\tilde{\Gamma}$ as generated by $\rho(S) \cup A$. Then the induced sofic approximation $(\sigma'_n \colon {\rm F}_S \to {\rm Sym}(m_n))_n$, when viewed as a sequence of $S$-edge-labeled graphs, comes endowed with an additional $A$-labelling on edges. This $A$-labeling measures the failure of $\rho$ being a homomorphism. In other terms, we obtain a map $\beta_n \colon S \to {\rm map}(\{1,\dots,m_n\}, A)$ measuring non-multiplicativity on the good part of the sofic approximation. On the other side $\alpha \in {\rm H}^2(\Gamma,A)$ is represented by a $2$-cocycle $\alpha' \colon R \to A$, where $R$ is a set of generating relations of $\Gamma$ and $\alpha(r)$ is defined to be $\tilde \pi(r) \in A,$ where $\tilde \pi \colon {\mathrm F}_S \to \tilde \Gamma$ is the extension of $\rho|_S$ to a homomorphism on ${\rm F}_S$. It is easy to see that $\delta_1(\beta_n)$ almost represents the image of $\alpha'$ in ${\rm map}(R,{\rm map}(\{1,\dots,m_n\},A))$. Indeed, $\delta_1(\beta_n)(r)$ is computed by adding up the errors along a relation $r \in R$. On the good part of the sofic approximation, this is just a way of computing $\alpha'(r).$ Taking the limit, we obtain $\theta^{\rm P}_*(\alpha)=0.$ This finishes the sketch of the argument.
\end{remark}

\subsection{Stability} \label{sec:stable}

By now there is a zoo of notions of stability of almost representations in permutations; including being \emph{stable} \cite{MR2500002}, \emph{weakly stable} \cite{MR3350728}, \emph{flexibly stable}, \emph{weakly flexibly stable} \cite{MR4027744}, \emph{very flexibly stable}, and \emph{weakly very flexibly stable} \cite{MR4134896}. Additional confusion arises since the notion of very flexible stability is weaker than flexible stability itself.

In retrospect and in view of the results of \cite{MR4027744}, the unflexible notions of stability are somewhat pathological and too rigid in order to be useful beyond the amenable case, see \cite{MR3999445}. The difference between weak and non-weak notions refers to the questions if one wants to assume from the start that the almost representations are almost free. Since we are only interested in sofic approximations this requirement is automatic in our setting.

Anyhow, in view of this discussion it needs a profound justification and linguistic skillfuless to introduce yet another notion of stability. Let's start with an observation:

\begin{lemma} \label{lem:flex}
A group $\Gamma$ is weakly flexibly stable if and only if the limit action of any sofic approximation is $\Gamma$-equivariantly internally isomorphic to a metric ultraproduct of finite $\Gamma$-actions.
\end{lemma}
\begin{proof}
This is a consequence of the results in \cite{MR919287}.
\end{proof}

The requirement that the isomorphism be internal is non-trivial, but leads us in a different direction and we will ignore this point.

\begin{definition} \label{def:stable}
We say that a group $\Gamma$ is \emph{stable in finite actions} if the limit action of any sofic approximation is weakly contained in the family of finite $\Gamma$-actions. Equivalently, the limit action of any sofic approximation of $\Gamma$ is contained in a metric ultraproduct of finite $\Gamma$-actions.
\end{definition}

It is clear from Lemma \ref{lem:flex} that any weakly flexibly stable group $\Gamma$ is stable in finite actions. Let us also note that any sofic group which is stable in finite actions must be residually finite. Using the definition of weak containment and following the discussion in \cite{MR4434538}, let us now characterize and give a criterion that implies stability in finite actions in more familiar terms. Spelling everything out, we obtain the following result:

\begin{lemma} Let $\Gamma$ be a countable discrete group. Consider the following conditions:
\begin{enumerate}[$(i)$]
    \item The group $\Gamma$ be stable in finite actions.
    \item For all $K \subset \Gamma$ finite, $d \in \mathbb N$, and $\varepsilon>0$, there exists $\delta>0$ and $F\subset \Gamma$ finite, such that for every $n \in \mathbb N$, every map $\varphi \colon \Gamma \to {\rm Sym}(n)$ satisfying
    $$d_n(\varphi(gh),\varphi(g)\varphi(h))<\delta \quad \forall g,h \in F$$ and
    $$d_n(1_n,\varphi(g))>1-\delta \quad \forall g \in F \setminus \{1_\Gamma\},$$
    and any partition $\{1,\dots,n\}=A_1 \cup \cdots \cup A_d$, there exists a finite set $X$, a homomorphism $\psi \colon \Gamma \to {\rm Sym}(X)$, and a partition $X=B_1\cup \cdots \cup B_d$ such that
    $$\left| \frac{|A_i \cap \varphi(g)A_j|}{n} - \frac{|B_i \cap \psi(g)B_j|}{|X|} \right| < \varepsilon, \quad \forall i,j \in \{1,\dots,d\}, g \in K.$$
    \item For all $K \subset \Gamma$ finite and $\varepsilon>0$, there exists $\delta>0$ and $F\subset \Gamma$ finite, such that for every $n \in \mathbb N$ and every map $\varphi \colon \Gamma \to {\rm Sym}(n)$ satisfying
    $$d_n(\varphi(gh),\varphi(g)\varphi(h))<\delta \quad \forall g,h \in F$$ and
    $$d_n(1_n,\varphi(g))>1-\delta \quad \forall g \in F \setminus \{1_\Gamma\},$$
    there exists $m \in \mathbb N$, a set $X$ of cardinality in $[nm,(1+\varepsilon)nm]$, a subset $X_0 \subseteq X$ of cardinality $nm$, a $m$-$1$-map $\pi \colon X_0 \to \{1,\dots,n\},$ and a homomorphism $\psi \colon \Gamma \to {\rm Sym}(X)$ such that
    $$|\{x \in X_0 \mid \pi(\psi(g)x) \neq \varphi(g) \pi(x) \}|< \varepsilon |X_0|, \quad \forall g \in K.$$
\end{enumerate}
Then, the following implications hold:
$$(i) \Leftrightarrow (ii) \Leftarrow (iii).$$
\end{lemma}

Condition $(iii)$ is a way of understanding a sufficient condition for stability in finite actions in terms of the notion of \emph{branched cover} (of an almost homomorphism) as studied by Ioana \cite{MR4434538}*{Definition 1.2}. In particular, a group is stable in finite actions if every sofic approximation has a branched cover which is equivalent to a sequence of finite actions, see \cite{MR4434538}*{Remark 1.4}. As pointed out to us by Adrian Ioana, it is likely that, as a consequence of results in \cite{MR4434538}, Condition $(iii)$  for Kazhdan groups might also imply flexible stability.
See Remark \ref{rem:discu} for examples of groups which are stable in finite actions and potential non-examples.
As a positive result, we can record the following inheritance property:
\begin{theorem} \label{thm:inh}
Let $\Gamma = \Lambda \rtimes H$, such that:
\begin{enumerate}[$(i)$]
\item The group $\Lambda$ is stable in finite actions.
\item[$(ii)$] Every finite index normal subgroup of $\Lambda$ contains a finite index subgroup, which is normalized by $H$. In particular, this holds if $\Lambda$ is finitely generated.
\item[$(iii)$] The group $H$ is residually finite and amenable.
\end{enumerate}
Then, the group $\Gamma$ is stable in finite actions.
\end{theorem}
\begin{proof} One can follow almost verbatim the argument on page 488 in \cite{MR2931406}, while noting that the restriction of the limit action of a sofic approximation of $\Gamma$ is the limit action of the induced sofic approximation of $\Lambda$, and that Problem ${\rm A}.4$ on page 505 of the same paper has been solved in \cite{MR3047068}.
\end{proof}

Thus, as a consequence of the preceding theorem, $\mathbb F_2 \times \mathbb Z$ is an example of a group which is stable in finite actions, but not flexibly stable, see \cite{MR4134896}.
\begin{remark}
Theorem \ref{thm:inh} should be compared to \cite{MR2931406}*{Conjecture 4.14} which was proved as \cite{MR3047068}*{Theorem 1.4}. 
There, it is not assumed that the extension of $H$ by $\Lambda$ is a crossed product. However, the argument on \cite{MR2931406}*{page 488} seems to be incomplete without further assumptions. Indeed, in the argument, it is assumed implicitly that a finite extension of a residually finite amenable group is again residually finite. While this is true for crossed products, counterexamples to the general case have been produced in \cite{erschler}.\end{remark}

Another line of argument can be extracted from the proof of  \cite{MR2077037}*{Theorem 2.8} by Lubotzky--Shalom. Recall, 
a group is termed {\rm LERF} (for locally extended residually finite) if every finitely generated subgroup is closed in the profinite topology

\begin{theorem} \label{thm:lubsh}
Let $\Gamma$ be a group and $\Lambda\leq \Gamma$ be a normal subgroup, such that:
\begin{enumerate}[$(i)$]
\item The group $\Lambda$ is stable in finite actions.
\item Every finitely generated subgroup of $\Lambda$ is closed in the profinite topology of $\Gamma$. In particular, this holds if $\Gamma$ is {\rm LERF}.
\item The group $\Gamma/\Lambda$ is amenable.
\end{enumerate}
Then, the group $\Gamma$ is stable in finite actions.
\end{theorem}
\begin{proof}
Let $\Gamma \curvearrowright ({\rm P},\mu)$ be the limit action of a sofic approximation of $\Gamma$. The, by $(i)$ the restriction to $\Lambda$ is weakly contained in finite $\Lambda$-actions. Thus, since the quotient $\Gamma/\Lambda$ is amenable by $(iii)$, the original action is weakly contained in the co-induction of finite $\Lambda$-actions using \cite{MR3047068}*{Theorem 1.1}. However, by $(ii)$ and \cite{MR2931406}*{Lemma 4.15}, any such action is weakly contained in the family of finite $\Gamma$-actions.
\end{proof}

Note that Theorem \ref{thm:lubsh} implies that surface groups are stable in finite actions. This is also a consequence of results in by Lazarovich--Levit--Minsky \cite{minsky} and the remark after Definition \ref{def:stable}.

\begin{question}
Is there a residually finite group, which is not stable in finite actions?
\end{question}
We believe the answer to this question should be positive; but it seems hard to come up with a concrete example, see also Remark \ref{rem:discu}.

\subsection{Proof of the main result}
\label{sec:proof}
Let us now state our main result precisely, prove it, and discuss some potential consequences.

\begin{theorem} \label{thm:main}
Let $\Gamma$ be a countable discrete group and $A$ be a finite abelian group. Assume the following conditions:
\begin{enumerate}[$(i)$]
    \item $\Gamma$ is residually finite and of finite type,
    \item $\Gamma$ satisfies a co-systolic inequality for all finite transitive actions with coefficients in $A$ and dimension $2$,
    \item there exists a central extension $\tilde \Gamma$ of $\Gamma$ by $A$ which is not residually finite.
\end{enumerate}
Then, the following implication holds:
If $\Gamma$ is stable in finite actions, then the group $\tilde \Gamma$ is not sofic.
\end{theorem}

This result should be compared with a result of Bowen-Burton \cite{MR4105530}, who proved that if ${\rm PSL}_d(\mathbb Z)$, for $d \geq 5$, is flexibly stable, then there exists a non-sofic group. The advantage of our result is a weakening of the assumption of flexible stability and also that the resulting candidate of a non-sofic group is a bit more concrete. Nevertheless, we will spend the next section discussing the existence of a group that satisfies Conditions $(i)$-$(iii)$ in Theorem \ref{thm:main}.

\begin{proof}[Proof of Theorem \ref{thm:main}] Let $\alpha \in {\rm H}^2(\Gamma,A)$ be the cohomology class classifying the central extension $\tilde \Gamma$. Assume that $\tilde \Gamma$ is sofic and let ${\rm Q}$ be the associated limit $\tilde \Gamma$-action and ${\rm P}\coloneqq {\rm Q}^A$ be the limit action of the associated sofic approximation of $\Gamma$, see Lemma \ref{lem:inv}. Then, on the one hand $\theta^{\rm P}_*(\alpha)=0$ by Theorem \ref{thm:equalzero}. On the other side, if $\Gamma$ is stable in finite actions, then ${\rm P}$ is weakly contained in finite actions and hence $\theta^{\rm P}_*(\alpha) \neq 0$ by Theorem \ref{thm:neqzero}. This is a contradition and finishes the proof.
\end{proof}

\begin{remark}
We want to comment on the vanishing and non-vanishing of $\theta_*^P(\alpha)$ for particular measured $\Gamma$-Boolean algebras. Recall that for a p.m.p.\ preserving $\Gamma$-action $(X,\mu)$, we denote by ${\rm M}(X,\mu)$ its measure algebra.
By a result of Ab{\'e}rt--Weiss \cite{abertweiss}, the Bernoulli action $\Gamma \curvearrowright (\{0,1\}^{\Gamma},\nu^{\otimes \Gamma})$ is weakly contained in any other free p.m.p.\ action. In particular, since $\Gamma$ is residually finite, it is weakly contained in the pro-finite action $\Gamma \curvearrowright (\hat \Gamma,\mu_{\rm Haar})$. Hence, $$\theta^{{\rm M}(\{0,1\}^{\Gamma},\nu^{\otimes \Gamma})}_*(\alpha) \neq 0$$ by Theorem \ref{thm:neqzero}. On the other side, we have $\theta_*^{{\rm M}(X,\mu)}(\alpha)=0$ for $\Gamma \curvearrowright X=\{0,1\}^{\tilde \Gamma}/A$ with its natural probability measure and $\Gamma$-action by Theorem \ref{thm:equalzero}. In particular, $\Gamma \curvearrowright X=\{0,1\}^{\tilde \Gamma}/A$ is not weakly contained in the family of finite actions of $\Gamma.$
\end{remark}

There is also a notion of soficity of a p.m.p.\ $\Gamma$-action on a standard probability space, see \cite{MR2566316}*{Definition 1.3}, which amounts to weak containment in the limit action of a sofic approximation of the group $\Gamma$. In view of the previous remark, we obtain the following corollary.

\begin{corollary}
If $\Gamma$ satisfies Conditions $(i)$-$(iii)$ of Theorem \ref{thm:main} and is stable in finite actions, then the p.m.p.\ action $\Gamma \curvearrowright \{0,1\}^{\tilde \Gamma}/A$ is not sofic.
\end{corollary}

\begin{remark} \label{rem:discu}
There are only few groups $\Gamma$ for which \emph{all} p.m.p.\ $\Gamma$-actions are weakly contained in the family of finite actions. Any such group is clearly stable in finite actions. These include free groups, surface groups, hyperbolic 3-manifold groups \cite{MR3047068} and residually finite amenable groups but likely no groups with Kazhdan's property (T), see the discussion in \cite{MR2931406}*{p.\ 465}.
In particular, it is known that ${\rm SL}_n(\mathbb Z)$ fails to have this property for $n \geq 3$, see \cite{MR2077037}. Thus, potentially, if any Kazhdan group can be shown to be stable in finite actions, then there exists a non-sofic p.m.p.\ action.
Another test case for stability in finite actions is the group $\mathbb F_2 \times \mathbb F_2$. By the recent refutation of the Connes' Embedding Conjecture \cite{mipstar}, there do exist p.m.p.\ actions of $\mathbb F_2 \times \mathbb F_2$ that are not weakly contained in finite actions. Again, any such action would be non-sofic if $\mathbb F_2 \times \mathbb F_2$ was stable in finite actions.
\end{remark}

\subsection{A candidate for a non-sofic group}
\label{Chapter:Candidate}

In this section, we will recall the construction of torsionfree lattices in ${\rm PSp}_{2d}(\mathbb Q_p)$ for $d \geq 4$ and $p$ large that satisfy all conditions in Theorem \ref{thm:main}.

Condition $(i)$ is satisfied by Lemma \ref{lem:bruhat} and Mal'cevs well-known theorem stating that any finitely generated subgroup of a linear group is residually finite. Condition $(ii)$ is satisfied by Theorem \ref{CorPSpExpansion} for $p$ large enough. It remains to discuss Condition $(iii)$, where we borrow the discussion from \cite{MR4080477}*{Section 5.2}. We also refer to \cite{MR4080477} for further references.

Let $D$ be the quaternion algebra over $\mathbb{Z}$ and $R$ be a commutative, unital ring. Then we can define the \emph{quaternion ring} $D_{R}$ over $R$ as
\begin{equation*}
D_{R} \coloneqq  R\langle i,j,k\rangle /\langle i^{2} = j^{2} = k^{2} = ijk=1_{R} \rangle .
\end{equation*}
There is an involution on $D_{R}$
\begin{equation*}
\tau_{R} \colon D_{R} \rightarrow D_{R} , ( r_{0} +r_{1} \cdot i +r_{2} \cdot j +r_{3} \cdot k ) \mapsto ( r_{0} -r_{1} \cdot i -r_{2} \cdot j-r_{3} \cdot k)
\end{equation*}
and therefore we can define the canonical hermitian sesquilinearform
\begin{equation*}
h_{n} \colon D_{R}^{n} \times D_{R}^{n} \rightarrow D_{R}, ((x_{1}, \hdots ,x_{n} ),(y_{1}, \hdots ,y_{n})) \mapsto \sum_{i=1}^{n} x_{i} \cdot \tau_{R} (y_{i} ).
\end{equation*}
We define the algebraic group
\begin{align*}
\mathbf{G} (R) &\coloneqq {\rm SU}_{n} ( D_{R},h_{n} ) \\
&\coloneqq \{ A \in D_{R}^{n\times n} \mid \det (A)=1, \forall x,y\in D_{R}^{n} : h_{n}( x,y) = h_{n}( A\cdot x, A\cdot y) \} ,
\end{align*}
i.e. the $n\times n$-matrices with entries in $D_{R}$ of determinant $1$ and whose associated linear maps preserve the form $h_{n}$.

Note that $\mathbf{G}$ is an absolutely almost simple, simply connected $\mathbb{Q}$-algebraic group of type $C_{n}$. We refer an interested reader to \cite{MR1278263} for general algebraic groups and \cite{MR4080477} for this particular one. 
First of all, note that $\mathbf{G}( \mathbb{Q}_{p})$ is isomorphic to ${\rm Sp}_{2n} ( \mathbb{Q}_{p})$ and this group is not compact. Recall also that 
\begin{equation}
 \mathbf{G} ( \mathbb{Z} [ \nicefrac{1}{p} ] ) \leq \mathbf{G} ( \mathbb{R} ) \times \mathbf{G} ( \mathbb{Q}_{p} ) \label{EqCandidateLatticeEmbedding}
\end{equation}
is a lattice, see \cite{MR1278263}*{Chapter 5.4}. Moreover, since $\mathbf{G}( \mathbb{R})$ is isomorphic to ${\rm U} ( 2n ) \cap {\rm Sp}_{2n}( \mathbb{C})$ and hence compact, it follows that  $\mathbf{G} ( \mathbb{Z} [ \nicefrac{1}{p} ] ) \leq \mathbf{G} ( \mathbb{Q}_{p} )$ is also a lattice embedding. Our choice of $\Gamma$ is a torsion-free subgroup of  $\mathbf{G} ( \mathbb{Z} [ \nicefrac{1}{p} ] )$ of finite index. Note that any such $\Gamma$ is also a torsionfree lattice in ${\rm PSp}_{2d}(\mathbb Q_p),$ since the center of ${\rm Sp}_{2d}(\mathbb Q_p)$ is finite.

We continue with a discussion of the universal central extension of $\Gamma$, which was first studied by Deligne \cite{MR507760} and later Prasad, who showed universality \cite{MR2020658}.

\begin{theorem}
For every prime $p\geq 5$ there is a unique central extension
\begin{equation}
0 \longrightarrow \mathbb{Z} /( p-1) \cdot \mathbb{Z} \overset{\iota}{\longrightarrow} \mathbf{G} ( \mathbb{Q}_{p} ) ' \overset{\pi}{\longrightarrow} \mathbf{G} ( \mathbb{Q}_{p} ) \longrightarrow 0 \label{EqThmExtensionG(Q_p)}
\end{equation}
which is universal in the sense that whenever $A$ is a discrete abelian group and there exists a central extension of topological groups
\begin{equation*}
0 \longrightarrow A \longrightarrow E \longrightarrow \mathbf{G} ( \mathbb{Q}_{p} ) \longrightarrow 0,
\end{equation*}
then there are continuous group homomorphisms such that the following diagram commutes
\begin{equation*}
\begin{tikzcd}
0 \ar[r] & A \ar[r] \ar[d] & E \ar[r] \ar[d] & \mathbf{G} ( \mathbb{Q}_{p} ) \ar[r] \ar[d, equal] & 0 \\
0 \ar[r] & \mathbb{Z} /( p-1) \cdot \mathbb{Z} \ar[r, "\iota"] & \mathbf{G} ( \mathbb{Q}_{p} ) ' \ar[r, "\pi"] & \mathbf{G} ( \mathbb{Q}_{p} ) \ar[r] & 0
\end{tikzcd} .
\end{equation*} \label{ThmExtensionG(Q_P)}
\end{theorem}

The pull-back of this central extension to the lattice $\Gamma \leq  \mathbf{G} ( \mathbb{Q}_{p} )$ is denoted by $\tilde{\Gamma} \coloneqq \pi^{-1} ( \Gamma )$.
The final ingredient is the following theorem, which provides Condition $(iii)$ of Theorem \ref{thm:main}.

\begin{theorem}[following Deligne, see \cite{MR4080477}]
Let $p\geq 5$ and $A$ the cyclic group of order $p-1$. For the torsion-free lattice $\Gamma \leq {\rm PSp}_{2d} (\mathbb{Q}_{p})$ discussed above, the central extension $\tilde{\Gamma}$
\begin{equation*}
0 \longrightarrow A \longrightarrow \tilde{\Gamma} \longrightarrow \Gamma \rightarrow 0 \label{EqExtGamma}
\end{equation*}
is not residually finite. \label{resfin}
\end{theorem}

Thus, taking $p$ large enough in order to allow for the application of Theorem \ref{CorPSpExpansion}, we see that all conditions of Theorem \ref{thm:main} are met.

\section*{Acknowledgments}

The ideas for this paper stem from 2018 and were communicated back then among various experts and were presented at a conference in Warwick in March 2019. The second named author thanks Alon Dogon for corrections and Adrian Ioana for a discussion on Section \ref{sec:stable}. We both thank Vadim Alekseev for inspiring discussions. The results of this paper also appear in the doctoral thesis of the first named author, see \cite{thesisgohla}.


\begin{bibdiv}
\begin{biblist}

\bib{abertweiss}{article}{
 Author = {Ab{\'e}rt, Mikl{\'o}s},
 author ={Weiss, Benjamin},
 Title = {Bernoulli actions are weakly contained in any free action},
 Journal = {Ergodic Theory Dyn. Syst.},
 Volume = {33},
 Number = {2},
 Pages = {323--333},
 Year = {2013},
}

\bib{AB}{book}{
   author={Abramenko, Peter},
   author={Brown, Kenneth S.},
   title={Buildings},
   series={Graduate Texts in Mathematics},
   volume={248},
   note={Theory and applications},
   publisher={Springer, New York},
   date={2008},
   pages={xxii+747},
}


\bib{MR3350728}{article}{
   author={Arzhantseva, Goulnara},
   author={P\u{a}unescu, Liviu},
   title={Almost commuting permutations are near commuting permutations},
   journal={J. Funct. Anal.},
   volume={269},
   date={2015},
   number={3},
   pages={745--757},
}

\bib{MR4027744}{article}{
   author={Becker, Oren},
   author={Lubotzky, Alexander},
   title={Group stability and Property (T)},
   journal={J. Funct. Anal.},
   volume={278},
   date={2020},
   number={1},
   pages={108298, 20},
}

\bib{MR3999445}{article}{
   author={Becker, Oren},
   author={Lubotzky, Alexander},
   author={Thom, Andreas},
   title={Stability and invariant random subgroups},
   journal={Duke Math. J.},
   volume={168},
   date={2019},
   number={12},
   pages={2207--2234},
}


\bib{borel}{book}{
   author={Borel, Armand},
   title={Linear algebraic groups},
   series={Graduate Texts in Mathematics},
   volume={126},
   edition={2},
   publisher={Springer-Verlag, New York},
   date={1991},
   pages={xii+288},
 
}

\bib{MR4105530}{article}{
   author={Bowen, Lewis},
   author={Burton, Peter},
   title={Flexible stability and nonsoficity},
   journal={Trans. Amer. Math. Soc.},
   volume={373},
   date={2020},
   number={6},
   pages={4469--4481},
}

\bib{MR3047068}{article}{
   author={Bowen, Lewis},
   author={Tucker-Drob, Robin D.},
   title={On a co-induction question of Kechris},
   journal={Israel J. Math.},
   volume={194},
   date={2013},
   number={1},
   pages={209--224},
}

\bib{MR1324339}{book}{
   author={Brown, Kenneth S.},
   title={Cohomology of groups},
   series={Graduate Texts in Mathematics},
   volume={87},
   note={Corrected reprint of the 1982 original},
   publisher={Springer-Verlag, New York},
   date={1994},
   pages={x+306},
}

\bib{MR4138908}{article}{
   author={Burton, Peter J.},
   author={Kechris, Alexander S.},
   title={Weak containment of measure-preserving group actions},
   journal={Ergodic Theory Dynam. Systems},
   volume={40},
   date={2020},
   number={10},
   pages={2681--2733},
}

\bib{chapman2023stability}{article}{
    title={Stability of Homomorphisms, Coverings and Cocycles II: Examples, Applications and Open problems}, 
    author={Michael Chapman and Alexander Lubotzky},
    year={2023},
    eprint={arXiv:2311.06706},
}
\bib{MR4080477}{article}{
   author={De Chiffre, Marcus},
   author={Glebsky, Lev},
   author={Lubotzky, Alexander},
   author={Thom, Andreas},
   title={Stability, cohomology vanishing, and nonapproximable groups},
   journal={Forum Math. Sigma},
   volume={8},
   date={2020},
}
\bib{MR507760}{article}{
   author={Deligne, Pierre},
   title={Extensions centrales non r\'{e}siduellement finies de groupes
   arithm\'{e}tiques},
   language={French, with English summary},
   journal={C. R. Acad. Sci. Paris S\'{e}r. A-B},
   volume={287},
   date={1978},
   number={4},
   pages={A203--A208},
}

\bib{dogon}{article}{
   author={Dogon, Alon},
   title={Flexible Hilbert-Schmidt stability versus hyperlinearity for
   property (T) groups},
   journal={Math. Z.},
   volume={305},
   date={2023},
   number={4},
   pages={Paper No. 58, 20},
}

\bib{MR2566316}{article}{
   author={Elek, G\'{a}bor},
   author={Lippner, G\'{a}bor},
   title={Sofic equivalence relations},
   journal={J. Funct. Anal.},
   volume={258},
   date={2010},
   number={5},
   pages={1692--1708},
}

\bib{erschler}{article}{
   author={Erschler, Anna},
   title={Not residually finite groups of intermediate growth,
   commensurability and non-geometricity},
   journal={J. Algebra},
   volume={272},
   date={2004},
   number={1},
   pages={154--172},
}

\bib{MR3536553}{article}{
   author={Evra, Shai},
   author={Kaufman, Tali},
   title={Bounded degree cosystolic expanders of every dimension},
   conference={
      title={STOC'16---Proceedings of the 48th Annual ACM SIGACT Symposium
      on Theory of Computing},
   },
   book={
      publisher={ACM, New York},
   },
   date={2016},
   pages={36--48},
}
\bib{MR2500002}{article}{
   author={Glebsky, Lev},
   author={Rivera, Luis Manuel},
   title={Almost solutions of equations in permutations},
   journal={Taiwanese J. Math.},
   volume={13},
   date={2009},
   number={2A},
   pages={493--500},
}


\bib{thesisgohla}{thesis}{
    author={Gohla, Lukas},
    title={Topics in Group Stability},
    type={Doctoral thesis at TU Dresden},
    year={2023},    
}





\bib{MR1694588}{article}{
   author={Gromov, Mikhail},
   title={Endomorphisms of symbolic algebraic varieties},
   journal={J. Eur. Math. Soc. (JEMS)},
   volume={1},
   date={1999},
   number={2},
   pages={109--197},
}

\bib{MR4134896}{article}{
   author={Ioana, Adrian},
   title={Stability for product groups and property ($\tau$)},
   journal={J. Funct. Anal.},
   volume={279},
   date={2020},
   number={9},
   pages={108729, 32},
}

\bib{MR4434538}{article}{
   author={Ioana, Adrian},
   title={On sofic approximations of $\mathbb{F}_2\times\mathbb{F}_2$},
   journal={Ergodic Theory Dynam. Systems},
   volume={42},
   date={2022},
   number={7},
   pages={2333--2351},
}

\bib{mipstar}{article}{
    title={${\rm MIP}^*={\rm RE}$}, 
    author={Zhengfeng Ji}, 
    author={Anand Natarajan}, 
    author={Thomas Vidick}, 
    author={John Wright},
    author={Henry Yuen},
    year={2020},
    eprint={arXiv:2001.04383},
}

  
\bib{kaufman2018cosystolic}{article}{
  author={Kaufman, Tali},
  author={Mass, David},
  title={Cosystolic Expanders over any Abelian Group.},
  journal={Electron. Colloquium Comput. Complex.},
  volume={25},
  pages={134},
  year={2018},
}


\bib{kechris}{book}{
   author={Kechris, Alexander S.},
   title={Global aspects of ergodic group actions},
   series={Mathematical Surveys and Monographs},
   volume={160},
   publisher={American Mathematical Society, Providence, RI},
   date={2010},
   pages={xii+237},
}

\bib{MR2931406}{article}{
   author={Kechris, Alexander S.},
   title={Weak containment in the space of actions of a free group},
   journal={Israel J. Math.},
   volume={189},
   date={2012},
   pages={461--507},
}

\bib{minsky}{article}{
   author={Nir Lazarovich},
   author={Arie Levit},
   author={Yair Minsky},
   title={Surface groups are flexibly stable},
   year={2019},
    eprint={arXiv:1901.07182},
}


\bib{MR3966743}{article}{
   author={Lubotzky, Alexander},
   title={High dimensional expanders},
   conference={
      title={Proceedings of the International Congress of
      Mathematicians---Rio de Janeiro 2018. Vol. I. Plenary lectures},
   },
   book={
      publisher={World Sci. Publ., Hackensack, NJ},
   },
   date={2018},
   pages={705--730},
   review={\MR{3966743}},
}

\bib{LMM}{article}{
   author={Lubotzky, Alexander},
   author={Meshulam, Roy},
   author={Mozes, Shahar},
   title={Expansion of building-like complexes},
   journal={Groups Geom. Dyn.},
   volume={10},
   date={2016},
   number={1},
   pages={155--175},
}

\bib{MR2077037}{article}{
   author={Lubotzky, Alexander},
   author={Shalom, Yehuda},
   title={Finite representations in the unitary dual and Ramanujan groups},
   conference={
      title={Discrete geometric analysis},
   },
   book={
      series={Contemp. Math.},
      volume={347},
      publisher={Amer. Math. Soc., Providence, RI},
   },
   date={2004},
   pages={173--189},
}



\bib{MR3755725}{article}{
   author={Oppenheim, Izhar},
   title={Local spectral expansion approach to high dimensional expanders
   Part I: Descent of spectral gaps},
   journal={Discrete Comput. Geom.},
   volume={59},
   date={2018},
   number={2},
   pages={293--330},
}

\bib{pestov}{article}{
   author={Pestov, Vladimir G.},
   title={Hyperlinear and sofic groups: a brief guide},
   journal={Bull. Symbolic Logic},
   volume={14},
   date={2008},
   number={4},
   pages={449--480},
}

\bib{MR1278263}{book}{
   author={Platonov, Vladimir},
   author={Rapinchuk, Andrei},
   title={Algebraic groups and number theory},
   series={Pure and Applied Mathematics},
   volume={139},
   note={Translated from the 1991 Russian original by Rachel Rowen},
   publisher={Academic Press, Inc., Boston, MA},
   date={1994},
   pages={xii+614},
}

\bib{MR2020658}{article}{
   author={Prasad, Gopal},
   title={Deligne's topological central extension is universal},
   journal={Adv. Math.},
   volume={181},
   date={2004},
   number={1},
   pages={160--164},
}

\bib{MR919287}{article}{
   author={Ross, David},
   title={Automorphisms of the Loeb algebra},
   journal={Fund. Math.},
   volume={128},
   date={1987},
   number={1},
   pages={29--36},
}

\bib{MR3966829}{article}{
   author={Thom, Andreas},
   title={Finitary approximations of groups and their applications},
   conference={
      title={Proceedings of the International Congress of
      Mathematicians---Rio de Janeiro 2018. Vol. III. Invited lectures},
   },
   book={
      publisher={World Sci. Publ., Hackensack, NJ},
   },
   date={2018},
   pages={1779--1799},
}


\bib{MR1803462}{article}{
   author={Weiss, Benjamin},
   title={Sofic groups and dynamical systems},
   note={Ergodic theory and harmonic analysis (Mumbai, 1999)},
   journal={Sankhy\={a} Ser. A},
   volume={62},
   date={2000},
   number={3},
}

\end{biblist}
\end{bibdiv}
\end{document}